\documentclass[a4paper,12pt]{article}
\usepackage{amssymb,amsthm,amsmath,amsfonts}
\usepackage{graphicx}
\usepackage{cite}

\usepackage{amssymb}
\usepackage{amsthm}
\usepackage{latexsym,array}
\usepackage{amsthm}
\usepackage{amsmath, mathrsfs}
\usepackage[english]{babel}
\usepackage{color}

\newtheorem{theorem}{Theorem}[section]
\newtheorem{definition}[theorem]{Definition}

\newtheorem{proposition}[theorem]{Proposition}
\newtheorem{corollary}[theorem]{Corollary}
\newtheorem{remark}[theorem]{Remark}

\newcommand{\rr}{\mathbb{R}}
\newcommand{\D}{\mathbb{D}}
\newcommand{\e}{{\mathbf{e}}}
\newcommand{\edag}{{\mathbf{e}^\dagger}}

\newcommand{\bc}{\mathbb{B}\mathbb{C}}

\newcommand{\R}{{\mathbb R}}
\newcommand{\C}{{\mathbb{C}}}
\newcommand{\n}{{\mathbb{N}}}

\newcommand{\li}{{\bf i}}
\newcommand{\ji}{{\bf j}}
\newcommand{\ki}{{\bf k}}

\begin{document}
\title{A bicomplex proportional fractional $(\vartheta,\varphi)-$weighted Cauchy-Riemann operator using Riemann-Liouville derivatives with respect to an hyperbolic-valued function}
\small{
\author {Jos\'e Oscar Gonz\'alez-Cervantes$^{(1,2)}$, Juan Adri\'an Ram\'irez-Belman$^{(2)}$ \\ and \\ Juan Bory-Reyes$^{(3)\footnote{corresponding author}}$}
\vskip 1truecm
\date{\small $^{(1)}$ Departamento de Matem\'aticas, ESFM-Instituto Polit\'ecnico Nacional. 07338, Ciudad M\'exico, M\'exico\\ Email: jogc200678@gmail.com\\
$^{(2)}$ SEPI, ESFM-Instituto Polit\'ecnico Nacional. 07338, Ciudad M\'exico, M\'exico\\ Email: adrianrmzb@gmail.com\\
$^{(3)}$ {SEPI, ESIME-Zacatenco-Instituto Polit\'ecnico Nacional. 07338, Ciudad M\'exico, M\'exico}\\Email: juanboryreyes@yahoo.com
}}

\maketitle
\begin{abstract}
Based on the Riemann-Liouville derivatives with respect to functions taking values in the set of hyperbolic numbers, we consider a novel bicomplex proportional fractional $(\vartheta,\varphi)-$weighted Cauchy-Riemann operator, involving weights hyperbolic orthogonal bicomplex functions. This operator is defined for the first time here, and its associated fractional Borel-Pompeiu formula is proved as the main result.
\end{abstract} 

\noindent
\textbf{Keywords.} Bicomplex analysis;  Proportional fractional integrals and derivatives  with respect to another function, Riemann-Liouville derivatives; Cauchy-Riemann operator; Borel-Pompeiu formula.\\
\textbf{MSC20220 Classification:} 26A33, 30E20, 30G35

\section{Introduction} 
Fractional calculus deals with non-integer order integration and differentiation, and it is a natural generalization to the classical calculus, which became a powerful and widely used tool for better modeling of some real world problems and control of processes in variety areas of applied science and engineering, see \cite{GM, KST, OS, O, P, SKM} for more detail about fractional operators, including their properties and applications. A brief historical survey of the theory along classical lines is addressed in \cite{MR, Ro}. 

During last decades, many researchers have proposed different ways of defining fractional derivative operators, ranging from the most recognized and strongly established Riemann-Liouville and Caputo to others recently arose.
 
Several articles considering the so-called proportional fractional integrals and derivatives of a function with respect to another function came to light. The list of contributions is quite long, such as \cite{A, KST, jarad2020more, jarad2017generalized, jarad2020more18, jarad2017new}, to cite just a few.

The algebra of bicomplex numbers was first introduced in 1892 by Corrado Segre \cite{S} and also studied in \cite{Sco, Spa, CC}. The
properties of holomorphic functions defined on bicomplex numbers goes back as far as the book of Price \cite{Pr}. Then, without pretense of completeness, we direct the attention of the reader to \cite{EMDA}, which takes into account the theory of bicomplex holomorphic functions of
bicomplex variables. 

Fractional bicomplex calculus is a very recent topic of research, see \cite{CTOP, ATV, BG} and the references therein. In \cite{BG}, inspired from the work \cite{ATV}, where a Borel-Pompieu formula induced by a complex $\psi$-weighted Cauchy-Riemann operator is derived, the authors developed a bicomplex $(\vartheta,\varphi)-$weighted framework of fractional Cauchy-Riemann operator.

We combine the proportional fractional calculus with respect to functions with the hypercomplex analysis of solutions of the bicomplex $(\vartheta,\varphi)-$weighted fractional Cauchy-Riemann operator. This research trend generalizes and strengthens the achievements of \cite{BG} in a different direction: the managing of the proportional derivation with respect to an hyperbolic-valued function.

\section{Preliminaries}
To make the presentation self contained, we first give a brief exposition of the basic concepts of fractional proportional operators and of the weighted bicomplex holomorphic functions employed throughout the work. The materials provided can be found in \cite{jarad2020more, jarad2017generalized, jarad2020more18, jarad2017new} and \cite{CTOP, ATV, BG} respectively.
\subsection{Proportional fractional integrals and derivatives of a function with respect to a certain function}
Consider the continuous functions $\chi_{0}, \chi_{1}:[0,1]\times \R \rightarrow [0,\infty)$ such that
\begin{align*}
    \lim_{\sigma \to 0^{+}}\chi_{1}(\sigma,t)=1, \lim_{\sigma \to 0^{+}}\chi_{0}(\sigma,t)=0, \lim_{\sigma \to 1^{-}}\chi_{1}(\sigma,t)=0, \lim_{\sigma \to 1^{-}}\chi_{0}(\sigma,t)=1.
\end{align*}
Given $\phi \in C^{1}(\rr)$ such that $\phi(t)>0$, for all $t\in \rr$, then the proportional derivative of $f\in C^{1}(\rr)$ with respect to $\phi$ of order $\sigma \in [0,1]$, is defined to be
\begin{equation}
    (D^{\sigma, \phi} f )(t):=\chi_{1}(\sigma,t)f(t)+\chi_{0}(\sigma,t) \frac{f'(t)}{\phi'(t)}.
\end{equation}

In particular, for $\chi_{1}(\sigma,t)=1-\sigma$ and $\chi_{0}(\sigma,t)=\sigma$ we have
\begin{equation}
    (D^{\sigma, \phi} f )(t)=(1-\sigma)f(t)+\sigma \frac{f'(t)}{\phi'(t)}.
\end{equation}
If $\phi(t)\equiv t,$ we drop the subscript $\phi$ on $D^{\sigma, \phi}$.

The proportional integral of $f$ with respect to $\phi$ and order $n\in \n$ reads
$$({_{a}}I^{n,\sigma,\phi} f)(t)=\frac{1}{\sigma^{n}\Gamma(n)}\int_{a}^{t}e^{\frac{\sigma-1}{\sigma}(\phi(t)-\phi(\tau))}(\phi(t)-\phi(\tau))^{n-1}f(\tau)\phi^{'}(\tau)d\tau,$$
where $\Gamma$ is the gamma function.

Set $\alpha \in \C$, with $0<\Re\alpha<1$. The left and right proportional fractional integrals of $f\in AC^{1}([a,b])$ with order $\alpha$, in the Riemann-Liouville setting are defined by
$$({_{a}}I^{\alpha, \sigma}f)(t)=\frac{1}{\sigma^{\alpha}\Gamma(\alpha)}\int_{a}^{t}e^{\frac{\sigma-1}{\sigma}(t-\tau)}(t-\tau)^{\alpha-1}f(\tau)d\tau$$ 
and
$$(I_{b}^{\alpha, \sigma}f)(t)=\frac{1}{\sigma^{\alpha}\Gamma(\alpha)}\int_{t}^{b}e^{\frac{\sigma-1}{\sigma}(\tau- t)}(\tau-t)^{\alpha-1}f(\tau)d\tau$$
respectively.

Important case is where $\sigma=1$ to have  
\begin{equation}\label{fract_RL_1}
({}_a I^{\alpha,1} f) (t)  = \frac{1}{  \Gamma(\alpha)} \int_a^t ( t-\tau ) ^{\alpha-1}f (\tau)  d\tau,
\end{equation}
which is the well-known left-fractional integral of $f$ of order $\alpha$ in the Riemann-Liouville sense.

The left and right proportional derivatives of $f$ are defined by
\begin{align*}
    ({_{a}}D^{\alpha,\sigma}f)(t)=D^{n,\sigma}{_{a}}I^{n-\alpha, \sigma}f(t)\quad \text{and} \quad (D_{b}^{\alpha,\sigma}f)(t)=D^{n,\sigma}I_{b}^{n-\alpha, \sigma}f(t) 
\end{align*}
where $D^{n,\sigma}:=\overbrace{D^{\sigma}\circ \cdots \circ D^{\sigma}}^{n-times}$, with $n=[\Re\alpha]+1$.  

The left fractional proportional integral with respect to $\phi$ is given by
\begin{align*}
    ({_{a}}I^{\alpha,\sigma,\phi}f)(t):=\frac{1}{\sigma^{\alpha}\Gamma(n)}\int_{a}^{t}e^{\frac{\sigma-1}{\sigma}(\phi(t)-\phi(\tau))}(\phi(t)-\phi(\tau))^{\alpha-1}f(\tau)\phi^{'}(\tau)d\tau,
\end{align*}
while the right fractional proportional integral reads
\begin{align*}
    (I_{b}^{\alpha,\sigma,\phi}f)(t):=\frac{1}{\sigma^{\alpha}\Gamma(n)}\int_{t}^{b}e^{\frac{\sigma-1}{\sigma}(\phi(\tau)-\phi(t))}(\phi(\tau)-\phi(t))^{\alpha-1}f(\tau)\phi^{'}(\tau)d\tau,
\end{align*}
The left and the right fractional proportional derivatives with respect to $\phi$ are
\begin{align}\label{equa11}
    ({_{a}}D^{\alpha,\sigma,\phi}f)(t)=D^{n,\sigma,\phi}{_{a}}I^{n-\alpha, \sigma,\phi}f(t)\  \text{and} \ (D_{b}^{\alpha,\sigma,\phi}f)(t)=D^{n,\sigma,\phi}I_{b}^{n-\alpha, \sigma,\phi}f(t) 
\end{align}
where $D^{n,\sigma,\phi}:=\underbrace{D^{\sigma,\phi}\circ \cdots \circ D^{\sigma,\phi}}_{n-times}$, with $n=[\Re\alpha]+1$.
Moreover, 
\begin{align}\label{TFProFracFunc}
    {_{a}}D^{\alpha,\sigma,\phi}\circ{_{a}}I^{\alpha, \sigma,\phi}f(t)=f(t)\quad \text{and} \quad D_{b}^{\alpha,\sigma,\phi}\circ I_{b}^{\alpha, \sigma,\phi}f(t)=f(t). 
\end{align}

\subsection{The Hausdorff derivative model}
The Hausdorff derivative has a definite statistical interpretation and many real-world applications to a wide range of problems, see for instance \cite{cai2018fractal,chen2017fractal, chen2006time, chenliang}.

The Hausdorff derivative is a local differential operator and its computational costs are far lower than for the global fractional derivative, see \cite{At, 2019hausdorff, Ji} for more details.
 
Let us look at the important natural situation of the distance covered by a particle moving with the fractal behaviors of delay and jump in porous media, which can be measured in fractal time by 
\begin{equation}\label{Fract_time}
l(\tau)=v\cdot (\tau-a)^{\alpha}
\end{equation}
where $v$ is the velocity, $l$ denotes the distance, $\tau$ and $a$ represent the current and initial time, and $\alpha$ is the time fractal dimensionality. 

We start with the variable velocity problem, the distance covered is given by
$l(t)=\int_{a}^{t}v(\tau)d(\tau-a)^{\alpha}$ 
and the velocity on fractal medium is defined by
\begin{equation}\label{Lim_Hff_dev}
\frac{dl}{d(t-a)^{\alpha}}=\lim_{t\to c}\frac{l(t)-l(c)}{(t-a)^{\alpha}-(c-a)^{\alpha}}=\frac{1}{\alpha(t-a)^{\alpha-1}}\frac{dl}{dt}
\end{equation}
where $t$ and $c$ represent the final and internal time respectively. 

The limit version of the Hausdorff derivative is given by \eqref{Lim_Hff_dev} and its general formulation is 
\begin{equation}\label{Lim_Hff_dev_G}
\frac{dl}{dt^{\alpha}}=\lim_{t\to c}\frac{l(t)-l(c)}{t^{\alpha}-c^{\alpha}}=\frac{1}{\alpha t^{\alpha-1}}\frac{dl}{dt}.
\end{equation}
\noindent
The only difference between \eqref{Lim_Hff_dev} and \eqref{Lim_Hff_dev_G} is the presence of initial time in \eqref{Lim_Hff_dev}. 

According to \eqref{Fract_time}, we can get a precise estimate of the location of the particle at time $\tau$ as
\begin{equation}\label{Fract_time_Loc}
l(\tau)=v(t-a)^{\alpha}-v(t-\tau)^{\alpha},
\end{equation}
where $t$ is the final time, the first term on the right-hand side represents the total movement distance, and the second term the movement distance between $t$ and $\tau$. Now we consider a variable velocity problem, which can be derived similarly to that before. The differential form of equation \eqref{Fract_time_Loc} can be established as
\begin{equation}\label{difeq_fract}
dl=-vd(t-\tau)^{\alpha}.
\end{equation}
The integral formulation of equation \eqref{difeq_fract} is stated by
\begin{equation}\label{fract_RL_0}
l(t)=\int_{a}^{t}-v(\tau)d(t-\tau)^{\alpha}.
\end{equation}
Note that \eqref{fract_RL_0} reveals to be identical with \eqref{fract_RL_1} except the constant before the integration.
\subsection{A complex $\psi$-weighted Cauchy-Riemann operator}  
Let $z,w\in \mathbb C$. A product can be defined for $z$ and $w$ by  
\begin{align*}
\langle z   ,  w  \rangle_{\mathbb C}:= \frac{1}{2}\left(  \overline z  w + \overline w z \right)=
\frac{1}{2}\left(  z   \overline w +  w  \overline z \right).
\end{align*}
Consider functions $\psi_0 , \psi_1  \in  C^1(\C,\C)$  such that  $\langle \psi_0   ,  \psi_1   \rangle_{\mathbb C} = 0 $ on $\C$.
Let us introduce one piece of notation: $\psi:=(\psi_0,\psi_1 )$, where $\psi_0 = p_0  + ip_1$ and $\psi_1 =  q_0+ iq_1$. Then, we get that $\langle \psi_0, \psi_1 \rangle_{\mathbb C} \equiv  0$ on $\C$ if and only if  $p_ 0   \psi_1  = - i q_1\psi_0$  on $\C$. 

From the above, the $\psi$-weighted Cauchy-Riemann operator is defined by
\begin{align*}
\mathcal D_{\psi} :=  \psi_0 \frac{\partial}{\partial x }  +   \psi_1 \frac{\partial}{\partial y }.   
\end{align*}
We recall the definition of $\psi-$weighted holomorphicity, more details in \cite [Definition 2.8]{ATV}.
\begin{definition}
Let $\Omega \subset \mathbb C$ a regular domain, we say that a continuously differentiable function  $f: \Omega \rightarrow \mathbb C$ is $\psi-$weighted holomorphic in $\Omega$ if and only if $f$ is in the kernel of $\mathcal D_{\psi}$.
\end{definition}

\noindent
Note that classic holomorphicity arises when $\psi_0=1 $ and $ \psi_1 = {i}$. 

We end this section by recalling two important results about some integral representation associated to $\psi$-weighted Cauchy-Riemann operator: $\psi-$weighted Gauss theorem and $\psi-$weighted Cauchy-Pompeiu integral formula, see \cite [Theorem 4.1, Theorem 5.1]{ATV}.
\begin{theorem} Let $\Omega \subset \C$ be a regular domain and let $f$, $\psi_0 = p_0 + { i}p_1$, $\psi_1 = q_0 + {i}q_1$ in  $C^1(\Omega, \C) \cap  C(\overline \Omega, \C)$. Then
\begin{align*}  
\int_\Omega\left( 
{\mathcal D}_{\psi} f(w) + ( \frac{\partial p_0}{\partial x }   +  \frac{\partial q_0}{\partial y }  )f + ( \frac{\partial p_1}{\partial x }  +  \frac{\partial q_1}{\partial y }  ) {  i} f
\right)
dxdy =
\int_{\partial \Omega}
 f d\rho_{\psi}(w), 
\end{align*}
where $w=x+iy$ and $d\rho_{\psi}(w):= \psi_0(w)dy - \psi_1(w)dx.$
\end{theorem}

\begin{theorem} \label{weighted Cauchy-Pompeiu}
Let $\Omega \subset \C$ be a regular domain, $f \in  C^1(\Omega,\C)\cap C(\Omega,\C)$ and $\psi_0, \psi_1\in \C$. For any $z\in \Omega$ the following $\psi-$weighted Cauchy-Pompeiu formula holds
\begin{align*}  f(z) c_{\psi} =
\int_{\partial \Omega}  
f( w )E_{\psi}( w , z)d\rho_{\psi}( w ) - \int_\Omega 
E_{\psi}( w ,z) {\mathcal D}_{\psi} f( w )dx d y,
\end{align*}
 where $ w  = x + {  i} y$,  
\begin{align*}
c_{\psi} = \sin^2(\alpha) \int_0^{2\pi}
\frac{\psi_0
     \cos \theta + \psi_1 \sin \theta}{
\left( \dfrac{\cos \theta}
{r_0^2} - \dfrac{
\sin \theta \cos \alpha}{r_0r_1}
\right)
 \psi_0 +
\left(
\dfrac{\sin \theta}
{r^2_1} - 
 \dfrac{ \cos \theta \cos \alpha}{
r_0r_1}
\right)
 \psi_1}
d\theta
\end{align*}
and $E_\psi( \cdot, \cdot )$ is the $\psi$-Cauchy-type kernel, see \cite[Section 3]{ATV}.  
\end{theorem}
\subsection{Basic review of the bicomplex function theory}
The algebra of bicomplex numbers, denoted by $\bc$, is generated by two commuting imaginary units $\li$ and $\ji$  such that 
$$\li^{2}=\ji^{2}=-1; \li\ji=\ji\li.$$
Let us introduce two isomorphic algebra to $\C$: 

The isomorphic algebras to $\C$: $\C(\li) := \{x + \li y : x, y \in {\R}\}$ and $\C(\ji) := \{x + \ji y : x, y \in {\R}\}$ occur at the same time in $\bc$. Throughout the paper $\C$ is understood as $\C(\li)$.  

We will write all the bicomplex numbers as $Z=z_{1}+ z_{2}\ji;\, z_{1}, z_{2}\in {\C(\li)}.$

The hyperbolic unit $\ki$ (a unit which squares to $1$) can be factor into a product of the units $\li$ and $\ji$.

The bicomplex algebra $\bc$ has two distinguished zero divisors $\{\e, \edag \}$ over $\C(\li)$ and $\C(\ji)$, which are 
\begin{displaymath}
\e = \frac{1}{2}(1 + \ki), \hspace{5mm} \edag = \frac{1}{2}(1 - \ki),
\end{displaymath}
\begin{displaymath}
\e \edag = 0, \e^2 = \e, (\edag)^2 = \edag; \e + \edag= 1, \e - \edag= {\ki}.
\end{displaymath}
The zero divisors $\{\e, \edag \}$ form another basis of $\bc$ so-called idempotent basis. The $\C(\li)-$idempotent representation of $Z\in \bc$ is given by 
$$Z=z_1\e+z_2\edag,$$
where $z_1,z_2\in\C.$
 
According to the above representation, the addition and multiplication  of bicomplex numbers can be realized component-wise. Indeed,
if $Z =z_1 \e + z_2 \edag $ and $W =w_1 \e + w_2 \edag $ are two bicomplex numbers, where $z_1, z_2, w_1, w_2 \in \C(\li)$, then
$$Z +W=(z_1+w_1) \e + (z_2+w_2) \edag , \ Z W=(z_1w_1) \e + (z_2w_2) \edag .$$
A special subalgebra of $\bc$ is the set of hyperbolic numbers, defined by
\[\D := \{\lambda_1 + {\bf k} \lambda_2 \ \mid \ \lambda_1, \lambda_2 \in \mathbb R \}.\]
We refer to \cite{EMDA} for more information on these numbers.  

For $Z,W\in \bc$ the notation $Z\preceq W$ means that $Z-W\in \D^{+}$, where  
\begin{displaymath}
\D^{+} := \{ \lambda_1\e + \lambda_2\edag \ |\ \lambda_1 \in \R^{+}\ \land\ \lambda_2 \in \R^{+}\}.
\end{displaymath}
It is relation $\preceq$ that makes a partial order over $\D$ allowable, see for instance \cite{EMDA}.

We will use this partial order to define the hyperbolic modulus of $Z=z_1\e +z_2\edag \in \bc$ as 
$$|Z|_{\ki}:=|z_1|\e + |z_2|\edag,$$
which induces the topology in $\bc$ of the bicomplex balls 
$$B(W, r) := \{Z \in \bc \ \mid \   \  |Z-W|_{\bf k}   \prec r \}$$
with $r\in \D^{+}$ centered at $W\in \bc$. 

Given $Z  = z_1 \e +  z_2\edag \in \bc$ we introduce the conjugation $Z^* =  \overline z_1 \e +  \overline  z_2\edag$, where $\overline z_1, \overline z_2$ are usual complex conjugates to $z_1, z_2 \in \C({\bf i})$.

Note that 
$$ZZ^*=Z^*Z=  |Z|^2_{\bf k}.$$

Set $Z = z_1\e +z_2 \edag $ and  $W = w_1\e +w_2 \edag \in \bc$, we have 
\begin{align*} \langle Z, W\rangle_{{\bf k}} & := \frac{1}{2}\left({Z}^* W +  {W}^* Z   \right) = \frac{1}{2}\left( Z {W}^*  + W {Z}^*  \right)  \\
& =\frac{1}{2} ( \overline{z}_1 w_1 + \overline{w}_1z_1 ) \e  + 
\frac{1}{2}( \overline{z}_2 w_2 + \overline{w}_2z_2 ) \edag \\
& = \langle z_1,  w_1\rangle_{\C({\bf i })}  \e  + 
\langle z_2,  w_2 \rangle_{\C({\bf i })}  \edag. \\
\end{align*} 

We now recall, in the same fashion as \cite{EMDA}, the basic definitions and results for the theory of $\mathbb{BC}$-holomorphic functions.
\begin{definition}
Let $\Omega\subset\mathbb{BC}$ be an open set. Fix $Z \in \Omega\subset\mathbb{BC}$ and let $F: \Omega\rightarrow\mathbb{BC}$. The limit
$$F'(Z):=\lim_{\Omega \ni W \to Z} \frac{F(W)-F(Z)}{W-Z},$$
when $W-Z$ is an invertible bicomplex number, if exists, is called the derivative of $F$ at the point $Z$.
\end{definition}
\begin{definition}
A function $F:\Omega\subset\mathbb{BC}\rightarrow\mathbb{BC}$ is said to be $\mathbb{BC}$-holomorphic in $\Omega$ if for every $Z\in \Omega$ the derivative $F'(Z)$ exists.
\end{definition}
By \cite[Theorem 7.6.4]{EMDA}, a bicomplex valued function $F = f_1\e + f_2\edag $ defined on a product-type domain, i.e. $\Omega=\Omega_1 \e + \Omega_2 \edag \subset \bc$, where $\Omega_1 , \Omega_2  \subset \C({\bf i})$ are domains, is $\bc$-holomorphic if and only if  
$$F(Z) = f_1(z_1)\e + f_2(z_2)\edag$$
at every $Z=z_1{\e} +z_2\edag \in \Omega$, where $z_l\in \Omega_l$ and  $f_l\in C^1(\Omega_l, \C({\bf i})) \cap Ker \dfrac{d}{d \overline z_l} $ for $l=1,2$.

We now follow \cite[Subsection 11.2]{EMDA} making the following assumptions: Let $\Omega$ be a domain in $\bc$, and let $\Lambda \in \Omega$ be a two-dimensional, simply connected, piecewise smooth surface with boundary $\gamma= \partial \Lambda \subset \Omega$. It is required that $\Lambda$ be parametrized by $R = R(u, v)$ with $R = R_1\e+R_2\edag$, being $R_1$ and $R_2$ parametrization, respectively, of simply connected domains $\Lambda_1$ and $\Lambda_2$ in $\C({\bf i})$. Furthermore, $\gamma$ is parametrized by the restriction of $R$ onto $\partial \Lambda$ given by $r = r_1\e + r_2\edag$, where $r_1$ and $r_2$ are the parametrizations of $\gamma_1 := \partial \Lambda_1$ and of $\gamma_2 := \partial \Lambda_2$, respectively. The curves $\gamma_1$ and $\gamma_2$ are require to be piecewise smooth, closed Jordan curves in $\C({\bf i})$. 

Let $F(Z) = f_1(z_1)\e+f_2(z_2)\edag$ a continuous bicomplex function on $\gamma$, where $f_l \in C^{1}(\Lambda_l, \C), l=1,2$, then a bicomplex integration is defined to be  
\begin{align*}
\int_{\gamma} F(Z) dZ := \e \int_{\gamma_1} f_1(z_1) d z_1 + \edag \int_{\gamma_2} f_2(z_2) d z_2, 
\end{align*} 
where $Z= z_1\e + z_2\edag $ and  $z_1,z_2\in \mathbb C({\bf i})$.

Similarly, consider
\begin{align*}
\int_{\Lambda} F(Z) dZ\wedge dZ^* :=
\left( \int_{\Lambda_1}f_1(z_1) d z_1 \wedge d\bar z_1
\right)\e   +\left( \int_{\Lambda_2} f_2(z_2) d z_2 \wedge d\overline z_2
\right)\edag. 
\end{align*}  
\begin{theorem}\label{BBPF}(Bicomplex Borel–Pompeiu formula) Let $F \in C^1(\Omega, \bc)$;  $F(Z) = f_1(z_1)\e+ f_2(z_2)\edag, \  Z = z_1\e+z_2\edag$, and let $\gamma$ and $\Lambda$ satisfy the above assumptions. Then for any $W \in \Lambda \setminus \gamma$, we have
\begin{align*}
F(W) = \frac{1}{2\pi{\bf i}} \int_{\gamma} \frac{F(Z)}{Z - W} dZ + \frac{1}{2\pi{\bf i}}\int_{\Lambda} \frac{\dfrac{\partial F}{\partial Z^*}}{Z - W}dZ \wedge dZ^*,
\end{align*}
where $Z = z_1\e + z_2\edag, \ z_1,z_2\in \mathbb C({\bf i})$ and
$\dfrac{\partial}{\partial Z^*} = \e \dfrac{\partial}{\partial \overline z_1} + \edag \dfrac{\partial}{\partial \overline z_2}.$
\end{theorem}
\begin{remark}
A discussion of bicomplex integration can be found in \cite [Section 2]{BPS} and the references given there. 
\end{remark}

\section{On the theory of bicomplex $(\vartheta,\varphi)-$weighted holomorphic functions}
Let $\vartheta (Z)= \vartheta_1(z_1) \e +  \vartheta_2(z_2)\edag $ and $ \varphi(Z) = \varphi_1(z_1)\e +\varphi_2(z_2)\edag$ 
of variable $Z=z_1\e +z_2\edag\in \bc$ such that $\vartheta_1, \vartheta_2, \varphi_1, \varphi_2\in  C^1(\C({\bf i}),\C({\bf i}))$. Then $\langle \vartheta, \varphi \rangle_{\bf k} \equiv 0$ on $\bc$ if and only if $\langle \vartheta_{\ell}, \varphi_{\ell} \rangle_{\C({\bf i })} \equiv 0$ on $\C({\bf i})$  for $\ell=1,2$. 

\begin{proposition}
Let $\vartheta_1 =p_{_{1,1}} + {\bf i} p_{_{1,2}} $, \  $\vartheta_2  =  p_{_{2,1}}  + {\bf i}p_{_{2,2}}$, \ $\varphi_1 = q_{_{1,1}}  + {\bf i} q_{_{1,2}}$ and $\varphi_2  =  q_{_{2,1}} + {\bf i} q_{_{2,2}}$. Then    
$\langle \vartheta  ,  \varphi  \rangle_{\bf k} = 0$ on  $\bc$ 
if and only if 
\begin{align*}
 (p_{_{1,2}} \e   + p_{_{2,2}} \edag)    \varphi 
= - {\bf i} \left(  (q_{_{1,1}} \e     +   q_{_{2,1}} \edag\right)  \vartheta  ,   
\textrm{on} \ \bc.
\end{align*}
\end{proposition}

\begin{definition}
Let a bicomplex variable $Z = z_1\e +z_2 \edag =  (x_1+ y_1{\bf  i })\e +(x_2+ y_2 {\bf i})\edag$ and define   
\begin{align*}
 \frac{\partial}{\partial Z_{\vartheta \varphi } } := &     (\vartheta_1 \frac{\partial  }{\partial {x_1} } +  \varphi_1  \frac{\partial  }{\partial {y_1} }) \e  +  ( \vartheta_2 \frac{\partial  }{\partial {x_2} } +  \varphi_2 \frac{\partial  }{\partial {y_2} }) \edag,
\end{align*}
which we can write explicitly as   
\begin{align*} 
\frac{\partial  }{\partial Z_{\vartheta  \varphi } }  =  & \left(  ( p_{_{1,1}}  \frac{\partial  }{\partial {x_1} } 
 + q_{_{1,1}}   \frac{\partial  }{\partial {y_1} } )
+ {\bf i} (  p_{_{1,2}}    \frac{\partial  }{\partial {x_1} } +
     q_{_{1,2}}  \frac{\partial  }{\partial {y_1} }) \right)  \e   \\ 
		 &  + 
	\left(    ( p_{_{2,1}}   \frac{\partial  }{\partial {x_2} }  + q_{_{2,1}}  \frac{\partial  }{\partial {y_2} })  + {\bf i}  (p_{_{2,2}}   \frac{\partial  }{\partial {x_2} }  + q_{_{2,2}}  \frac{\partial  }{\partial {y_2} })  \right) \edag.
\end{align*}
We call $\dfrac{\partial  }{\partial Z_{\vartheta \varphi } }$ the bicomplex $(\vartheta,\varphi)-$weighted Cauchy-Riemann operator.
\end{definition} 
 
For the following result the $(\vartheta,\varphi)-$ weighted Gauss theorem for bicomplex functions is expressed out.
\begin{theorem}\label{WeightedTBPBC} 
Let $\Omega= \Omega_1\e+\Omega_2 \edag \subset \bc$ such that  $\Omega_1,\Omega_2\subset \mathbb C({\bf i})$ are domains 
and set $F =f_1 \e + f_2 \edag$ with $f_l\in C^1(\Omega_l, \C) \cap  C(\overline \Omega_l, \C) $ for $l=1,2$.
Set $\Lambda \subset \Omega$ a surface with smooth boundary $\gamma= \partial \Lambda \subset \Omega$ according to Theorem \ref{BBPF}.  In \cite{BG} we see the following results: 
\begin{equation*}   
 \int_{\Lambda} \left(  \frac{\partial F }{\partial Z_{\vartheta  \varphi } } 
 +  A_{\vartheta  \varphi }  F  + B_{\vartheta  \varphi }  {\bf i} F  \right)
dZ \wedge dZ^{*}  = \int_{\gamma} F(Z) d\rho_{\vartheta  \varphi }(Z) ,
\end{equation*}
where   $d\rho_{\vartheta  \varphi }(Z) =  d\rho_{\vartheta }(z_1) \e + d\rho_{\varphi }(z_1) \edag$ and
\begin{align*}
A_{\vartheta  \varphi } =&   (\frac{\partial  p_{_{ 1,1}}  }{\partial  x_1} +\frac{\partial  q_{_{ 1,1}} }{\partial  y_1} ) \e  + (\frac{\partial  p_{_{ 2,1}}  }{\partial  x_2}+\frac{ \partial  q_{_{ 2,1}}}{\partial  y_2} ) \edag  
\\
B_{\vartheta  \varphi } = &  (
\frac{\partial  p_{_{ 1,2}} }{\partial  x_1} + \frac{\partial  q_{_{ 1,2}} }{\partial  y_1}   ) \e + 
 ( \frac{\partial  p_{_{ 2,2}}  }{\partial  x_2}  + \frac{\partial  q_{_{ 2,2}}  }{\partial  y_2}
) \edag  \end{align*}
\end{theorem}
The following theorem reveals a bicomplex $(\vartheta,\varphi)-$Borel-Pompieu formula.
\begin{theorem}\label{weighted_Borel-Pompieu_Bicomplex} 
Under the same hypothesis of Theorem \ref{WeightedTBPBC} with $W\in \Lambda$ we get  
\begin{align*}  & F(W)c_{(\vartheta,\varphi)}  =
\int_{\gamma  }  
F (Z )E_{(\vartheta, \varphi)} (Z , W )d\rho_{(\vartheta, \varphi)}(Z ) 
 - \int_{\Lambda} 
E_{(\vartheta , \varphi)}(Z ,W ) \frac{\partial F }{\partial Z_{(\vartheta,\varphi)}}  dZ\wedge dZ^{*},
\end{align*}
where $c_{(\vartheta,\varphi)} := c_{\vartheta_1 \varphi_1 }  \e +  c_{\vartheta_2 \varphi_2}\edag,$ 
$E_{(\vartheta  , \varphi)} (Z,W) :=  E_{\vartheta_1\varphi_1 } (z_1 , w_1 ) \e  + E_{\vartheta_2\varphi_2 } (z_2 , w_2)\edag$ 
and $d\rho_{(\vartheta , \varphi)}(Z) :=  d\rho_{\vartheta_1\varphi_1 }(z_1 ) \e  + d\rho_{\vartheta_2 \varphi_2 }(z_2 ) \edag.$
\end{theorem}
\subsection{Bicomplex $(\vartheta,\varphi)-$weighted $(\vec{\alpha},\sigma,\phi)-$Cauchy-Riemann type operator}
Set ${P}=(a_1, c_1,a_2, c_2) , \ Q=(b_1, d_1,b_2, d_2) \in \mathbb R^4$ such that  $a_l< b_l$  y $c_l < d_l$ for $l=1,2 $ and let 
\begin{align*}
J_P^Q & := \{ z_1\e + z_2 \edag \in \bc \ \mid  \ \Re z_\ell  \in [a_\ell  ,b_\ell ], \ \Im z_\ell \in [c_\ell,d_\ell ] , \  \ \ell=1,2  \} \\
 &= \left( [a_1,b_1]+ {\bf i} [c_1,d_1]\right) \e + \left( [a_2,b_2]+ {\bf i} [c_2,d_2]\right) \edag  
\end{align*}
Let  $\vec{\alpha}=(\alpha_0, \alpha_1,\alpha_2,\alpha_3)$ and $\vec{\sigma}= (\sigma_0, \sigma_1,\sigma_2,\sigma_3)$ vectors of $ (0,1)^4$. Define 
$\phi=\phi_1 \e + \phi_2 \edag$ a product type function in $C^1(J_P^Q, \mathbb D^+)$, where $\phi_\ell : [a_\ell,b_\ell]\times [c_\ell,d_\ell]   \to \mathbb R^+ $, $ \ell =1,2$ such that $\dfrac{\partial \phi  }{\partial x_\ell}, \dfrac{\partial\phi }{\partial y_\ell}> 0$ on $ [a_\ell,b_\ell]\times [c_\ell,d_\ell] $, if $z_\ell=(x_\ell + {\bf i}y_\ell )$ for $\ell=1,2$.  

Consider 
$$ D \phi = (\frac{\partial \phi }{\partial {x_1} } +     \frac{\partial \phi }{\partial {y_1} }) \e  +  ( \frac{\partial \phi }{\partial {x_2} } +   \frac{\partial  \phi}{\partial {y_2} }) \edag$$
on $J_P^Q$.
\begin{definition} \label{Def1} Let $W= w_1\e + w_2\edag \in J_P^Q$. We shall write $F: J_P^Q  \to \bc\in AC^1(J_P^Q , \bc)$ if $F$ is a product type function, $F(Z)=F(z_1\e+ z_2 \edag )= f_1(z_1)\e + f_2(z_2)\edag $ with $f_\ell: [a_\ell ,b_\ell ]+ {\bf i} [c_\ell,d_\ell ] \to \mathbb C({\bf i}), \ \ell =1,2$ for which the maps $x_\ell \to f_\ell (x_\ell + {\bf i} \Im w_\ell)$ and  $y_\ell \to f_\ell (\Re w_\ell + {\bf i} y_\ell)$ belong  respectively to $AC^1([a_\ell,b_\ell], \C({\bf i}))$ and to $AC^1( [c_\ell,d_\ell], \C({\bf i}))$ for $\ell=1,2$. 

Also define: 
\begin{align*}
I_{a^+}^{1-\vec{\alpha}, \vec {\sigma}, \phi} F(Z,W): =&\left(I_{a_{1}^{+}}^{1-\alpha_{0},\sigma_{0},\phi_{0}}f_{1}(x_{1}+\li \Im w_{1}) + I_{c_{1}^{+}}^{1-\alpha_{1}, \sigma_{1},\phi_{1}}f_{1}(\Re w_{1}+\li y_{1})\right)\e +\\
& \left(I_{a_{2}^{+}}^{1-\alpha_{2},\sigma_{2},\phi_{2}}f_{2}(x_{2}+\li \Im w_{2}) + I_{c_{2}^{+}}^{1-\alpha_{3}, \sigma_{3},\phi_{3}}f_{2}(\Re w_{2}+\li y_{2})\right)\edag ,\\ 
D_{a^+}^{1-\vec {\alpha},\vec{ \sigma}, \phi}  F(Z,W): =&\left(D_{a_{1}^{+}}^{1-\alpha_{0},\sigma_{0},\phi_{0}}f_{1}(x_{1}+\li \Im w_{1}) + D_{c_{1}^{+}}^{1-\alpha_{1}, \sigma_{1},\phi_{1}}f_{1}(\Re w_{1}+\li y_{1})\right)\e +\\
& \left(D_{a_{2}^{+}}^{1-\alpha_{2},\sigma_{2},\phi_{2}}f_{2}(x_{2}+\li \Im w_{2}) + D_{c_{2}^{+}}^{1-\alpha_{3}, \sigma_{3},\phi_{3}}f_{2}(\Re w_{2}+\li y_{2})\right)\edag ,\\
I_{b^-}^{1-\vec {\alpha},\vec{ \sigma}, \phi}  F(Z,W) :=&\left(I_{b_{1}^{-}}^{1-\alpha_{0},\sigma_{0},\phi_{0}}f_{1}(x_{1}+\li \Im w_{1}) + I_{d_{1}^{-}}^{1-\alpha_{1}, \sigma_{1},\phi_{1}}f_{1}(\Re w_{1}+\li y_{1})\right)\e +\\
&\left(I_{b_{2}^{-}}^{1-\alpha_{2},\sigma_{2},\phi_{2}}f_{2}(x_{2}+\li \Im w_{2}) + I_{d_{2}^{-}}^{1-\alpha_{3}, \sigma_{3},\phi_{3}}f_{2}(\Re w_{2}+\li y_{2})\right)\edag ,\\
D_{b^-}^{1-\vec {\alpha},\vec{ \sigma}, \phi}  F(Z,W) :=&\left(D_{b_{1}^{+}}^{1-\alpha_{0},\sigma_{0},\phi_{0}}f_{1}(x_{1}+\li \Im w_{1}) + D_{d_{1}^{-}}^{1-\alpha_{1}, \sigma_{1},\phi_{1}}f_{1}(\Re w_{1}+\li y_{1})\right)\e +\\
& \left(D_{b_{2}^{-}}^{1-\alpha_{2},\sigma_{2},\phi_{2}}f_{2}(x_{2}+\li \Im w_{2}) + D_{d_{2}^{-}}^{1-\alpha_{3}, \sigma_{3},\phi_{3}}f_{2}(\Re w_{2}+\li y_{2})\right)\edag \end{align*}
and
\begin{align*} 
& R^{\vec {\alpha},\vec{ \sigma}, \phi}  F(Z,W)  
 := \\
&  \left(    I_{c_{1}^{+}}^{1-\alpha_{1}, \sigma_{1},\phi_{1}}f_{1}(\Re w_{1}+\li y_{1})D_{a_{1}^{+}}^{1-\alpha_{0},\sigma_{0},\phi_{0}} [1] + I_{a_{1}^{+}}^{1-\alpha_{0},\sigma_{0},\phi_{0}}f_{1}(x_{1}+\li \Im w_{1})  D_{c_{1}^{+}}^{1-\alpha_{1}, \sigma_{1},\phi_{1}}[1]      
 \right)  \e \\
   &+\left(  I_{a_{2}^{+}}^{1-\alpha_{2},\sigma_{2},\phi_{2}}f_{2}(x_{2} + 
  +\li \Im w_{2})D_{c_{2}^{+}}^{1-\alpha_{3}, \sigma_{3},\phi_{3}}[1]     +     I_{c_{2}^{+}}^{1-\alpha_{3}, \sigma_{3},\phi_{3}}f_{2}(\Re w_{2}+\li y_{2})D_{a_{2}^{+}}^{1-\alpha_{2},\sigma_{2},\phi_{2}}[1]\right) \edag ,
\end{align*}
where $\phi_0 (t) := \phi ((t+{\bf i} \Im w_1)\e + w_2 \edag )$ for $t\in [a_1,b_1]$,  $\phi_1 (t) := \phi ((\Re w_1 +{\bf  i} t)\e + w_2 \edag )$ for $t\in [c_1,d_1]$, $\phi_2(t) := \phi ( w_1 \e  +  (t+{\bf i} \Im w_2) \edag )$ for $t\in [a_2,b_2]$ and $\phi_3 (t) := \phi ( w_1 \e  +  (\Re w_2+{\bf i} t) \edag )$ for $t\in [c_2,d_2]$.
\end{definition}

\begin{proposition}\label{FTBC}
Let $F: J_P^Q  \to \bc\in AC^1(J_P^Q , \bc)$. Then, we have 
\begin{align}\label{equa9} 
& D_{a^+}^{1-\vec \alpha,\vec  \sigma, \phi} \circ I_{a^+}^{1-\vec \alpha,\vec  \sigma, \phi}  F(Z,W)= \nonumber \\
 &\left( f_{1}(x_{1}+\li \Im w_{1}) +   f_{1}(\Re w_{1}+\li y_{1})       
 \right)  \e +  \left( f_{2}(x_{2}    +\li \Im w_{2}) +  f_{2}(\Re w_{2}+\li y_{2}) \right) \edag \nonumber \\
  &  +  R^{\vec  \alpha, \vec \sigma, \phi}  F(Z,W).
\end{align}
 \end{proposition}
\begin{proof} The proof is based on a direct calculation using \eqref{TFProFracFunc}.
\end{proof}
To shorten notation we continue to write $\sigma$ for $(\sigma_0 +  {\bf i} \sigma_1)\e +  (\sigma_2 +  {\bf i} \sigma_3)\edag$.

Now, the bicomplex $(\vartheta,\varphi)-$weighted $(\vec{\alpha},\sigma,\phi)-$Cauchy-Riemann type operator can be introduced
\begin{definition} Let $W= w_1\e + w_2\edag \in J_{P}^Q$ fixed and  let $F\in AC^1(J_P^Q , \bc)$. The right bicomplex $(\vartheta,\varphi)-$weighted $(\vec{\alpha},\sigma,\phi)-$Cauchy-Riemann type operator, i.e., a   bicomplex weighted proportional fractional, of order $\vec \alpha$ and proportion $\sigma$ Cauchy-Riemann type operator associated to $\phi$  with weight $(\vartheta,\varphi)$ is defined by 
 \begin{align*}
 \frac{\partial ^{
\vec  \alpha,\sigma,\phi}F(Z,W)}{\partial Z_{\vartheta \varphi, a^+ } } := & (1-\sigma)  (I_{a^+}^{1-\vec \alpha,\vec  \sigma, \phi}  F)(Z,W) +  \sigma \dfrac{\displaystyle \frac{\partial}{\partial Z_{\vartheta \varphi } }   ( I_{a^+}^{1-\vec \alpha,\vec  \sigma, \phi} F)
(Z,W)   }{D \phi (Z)}.
 \end{align*}
Meanwhile, The left bicomplex $(\vartheta,\varphi)-$weighted $(\vec{\alpha},\sigma,\phi)-$Cauchy-Riemann type operator is 
\begin{align*}
 \frac{\partial ^{
\vec \alpha,\sigma,\phi}F(Z,W)}{\partial Z_{\vartheta \varphi, b^- }}:= & (1-\sigma)  (I_{b^-}^{1-\vec \alpha, \vec \sigma, \phi} F)(Z,W) +
   \sigma \dfrac{\displaystyle \frac{\partial}{\partial Z_{\vartheta \varphi } }   ( I_{b^-}^{1-\vec \alpha,\vec \sigma, \phi} F)
(Z,W) }{D \phi (Z)}.
\end{align*}
\end{definition}
Note that the previous  fractional   operators    preserve the  structure of that given in \eqref{equa11} for $n=1$.
\begin{proposition} Let $W= w_1\e + w_2\edag \in J_{P}^Q$ fixed and  let $F\in AC^1(J_P^Q , \bc)$. Let  $\lambda_\ell: [a_\ell ,b_\ell]\times [c_\ell,d_\ell]   \to \mathbb R $ for  $\ell=1,2$ be  such that 
\begin{align*}
(\vartheta_1 \frac{\partial \lambda_1 }{\partial {x_1} } +  \varphi_1  \frac{\partial  \lambda_1 }{\partial {y_1} }) \e  +  ( \vartheta_2 \frac{\partial  \lambda_2 }{\partial {x_2} } +  \varphi_2 \frac{\partial  \lambda_2 }{\partial {y_2} }) \edag  
=    (  D \phi (Z) ) \sigma^{-1} (1-\sigma) .
\end{align*}
Denote $\lambda_l(x_l,y_l)= \lambda(z_l)$ for $l=1,2$. Then
\begin{align}\label{equa10}
 \frac{\partial ^{
 \vec \alpha,\sigma,\phi}F(Z,W)}{\partial Z_{\vartheta \varphi, a^+ } } = &
   ( e^{ -  \lambda_1 } \e +
e^{  - \lambda_2 } \edag )  (D \phi (Z) )^{-1} \sigma  \frac{\partial}{\partial Z_{\vartheta \varphi } } \left[   ( e^{   \lambda_1 } \e +
e^{   \lambda_2 } \edag )      ( I_{a^+}^{1-\vec \alpha, \vec \sigma, \phi} F)
(Z,W)         \right] ,   
\end{align}  
\begin{align*}
 \frac{\partial ^{
 \vec \alpha,\sigma,\phi}F(Z,W)}{\partial Z_{\vartheta \varphi, b^- } } = &
   ( e^{ -  \lambda_1 } \e +
e^{  - \lambda_2 } \edag )  ( D \phi (Z) )^{-1} \sigma  \frac{\partial}{\partial Z_{\vartheta \varphi } } \left[   ( e^{   \lambda_1 } \e +
e^{   \lambda_2 } \edag )      ( I_{b^-}^{1-\vec \alpha, \vec \sigma, \phi}  F)
(Z,W)         \right] .   
\end{align*}
\end{proposition}
\begin{proof}
The proof starts with the observation that 
\begin{align*}
& (  D \phi (Z)  ) \sigma^{-1}  \frac{\partial ^{
\vec  \alpha,\sigma,\phi}F(Z,W)}{\partial Z_{\vartheta \varphi, a^+ } }= \\
&  ( D \phi (Z) ) \sigma^{-1} (1-\sigma)  (I_{a^+}^{1-\vec \alpha,\vec  \sigma, \phi} F)(Z,W) +    \frac{\partial}{\partial Z_{\vartheta \varphi } }   ( I_{a^+}^{1-\vec  \alpha,\vec  \sigma, \phi}  F)
(Z,W)    
 \end{align*}
and that
\begin{align*}
 \frac{\partial  ( e^{\lambda_1 } \e + e^{\lambda_2} \edag ) }{\partial Z_{\vartheta \varphi}} =  &   \left[(\vartheta_1 \frac{\partial \lambda_1 }{\partial {x_1} } +  \varphi_1  \frac{\partial  \lambda_1 }{\partial {y_1} }) \e  +  ( \vartheta_2 \frac{\partial  \lambda_2 }{\partial {x_2} } +  \varphi_2 \frac{\partial  \lambda_2 }{\partial {y_2} }) \edag\right]   ( e^{   \lambda_1 } \e +
e^{   \lambda_2 } \edag )  \\ 
= &   ( D \phi (Z)  ) \sigma^{-1} (1-\sigma)  ( e^{   \lambda_1 } \e +
e^{   \lambda_2 } \edag ) .
\end{align*}
Then
\begin{align*}
& \frac{\partial \left[( e^{\lambda_1 } \e + e^{\lambda_2} \edag)(I_{a^+}^{1-\vec \alpha,\vec  \sigma, \phi} F)
(Z,W)\right]}{\partial Z_{\vartheta \varphi }} =\\
& \left[\frac{\partial (e^{\lambda_1 } \e + e^{\lambda_2 } \edag )}{\partial Z_{\vartheta \varphi}} \right] (I_{a^+}^{1-\vec  \alpha,\vec  \sigma, \phi}  F)(Z,W) +  (e^{\lambda_1 } \e + e^{\lambda_2 } \edag ) \frac{\partial \left[( I_{a^+}^{1-\vec \alpha,\vec  \sigma, \phi}  F)(Z,W) \right]}{\partial Z_{\vartheta \varphi}}   =\\
&( e^{   \lambda_1 } \e + e^{   \lambda_2 } \edag )  \left\{   ( D \phi (Z)  ) \sigma^{-1} (1-\sigma)   ( I_{a^+}^{1-\vec \alpha,\vec  \sigma, \phi} F)(Z,W) + \frac{\partial \left[   ( I_{a^+}^{1-\vec \alpha,\vec  \sigma, \phi} F)
(Z,W)\right]}{\partial Z_{\vartheta \varphi}} \right\} =\\
& ( e^{   \lambda_1 } \e + e^{   \lambda_2 } \edag )  (  D \phi (Z)  ) \sigma^{-1}  \frac{\partial ^{\vec  \alpha,\sigma,\phi}F(Z,W)}{\partial Z_{\vartheta \varphi, a^+}}.   
\end{align*}
Therefore, 
\begin{align*}
 & \frac{\partial ^{
\vec  \alpha,\sigma,\phi}F(Z,W)}{\partial Z_{\vartheta \varphi, a^+ } } = (e^{-\lambda_1 } \e + e^{  - \lambda_2 } \edag )  (D \phi (Z) )^{-1} \sigma  \frac{\partial \left[   ( e^{   \lambda_1 } \e + e^{   \lambda_2 } \edag) ( I_{a^+}^{1-\vec \alpha,\vec  \sigma, \phi} F) (Z,W) \right]}{\partial Z_{\vartheta \varphi}}.   
\end{align*}
The second identity may be handled in much the same way.
\end{proof}
\section{Bicomplex $(\vartheta,\varphi)-$weighted $(\vec{\alpha},\sigma,\phi)-$Borel-Pompieu formula} 
We first prove a $(\vartheta,\varphi)-$weighted $(\vec{\alpha},\sigma,\phi)-$Gauss theorem for bicomplex functions
\begin{theorem} \label{weighted_Gauss_Bicomplex2}
Let $W= w_1\e + w_2\edag \in J_P^Q$ fixed and let $F\in AC^1( J_P^Q , \bc) $ be such that 
$$( e^{   \lambda_1 } \e + e^{   \lambda_2 } \edag )      ( I_{a^+}^{1-\vec \alpha, \vec \sigma , \phi}  F)
(\cdot,W)    \in C^1(J_P^Q , \bc) \cap  C(\overline{J_P^Q}, \bc).$$
Let $\Lambda \subset J_P^Q$ a surface with boundary $\gamma \subset J_P^Q$ under the assumptions of Theorem \ref{WeightedTBPBC}. Then 
\begin{align*}   
& \int_{\gamma}   ( e^{   \lambda_1 } \e +
e^{   \lambda_2 } \edag )      ( I_{a^+}^{1-\vec \alpha,\vec  \sigma, \phi}  F)
(Z,W)       d\rho_{\vartheta  \varphi }(Z) 
=  \int_{\Lambda} \left[    ( e^{   \lambda_1 } \e +
e^{    \lambda_2 } \edag )  ( D \phi (Z) ) \sigma^{-1}
\frac{\partial ^{
\vec  \alpha,\sigma,\phi}F(Z,W)}{\partial Z_{\vartheta \varphi, a^+ } } 
 \right. \\
&  \left. +  A_{\vartheta  \varphi }   ( e^{   \lambda_1 } \e +
e^{   \lambda_2 } \edag )      ( I_{a^+}^{1-\vec \alpha, \vec \sigma, \phi}  F)
(Z,W)         + B_{\vartheta  \varphi }  {\bf i}   ( e^{   \lambda_1 } \e +
e^{   \lambda_2 } \edag )      ( I_{a^+}^{1-\vec \alpha,\vec  \sigma, \phi}  F)
(Z,W)         \right]
dZ \wedge dZ^{*} .
\end{align*}
\end{theorem} 
\begin{proof}
By Theorem \ref{WeightedTBPBC} and considering the function $Z\mapsto (e^{\lambda_1 } \e + e^{\lambda_2 } \edag ) ( I_{a^+}^{1-\vec \alpha, \vec \sigma, \phi}  F)(Z,W)$ we see that 
\begin{align*}   
& \int_{\gamma}   ( e^{   \lambda_1(z_1) } \e +
e^{   \lambda_2 (z_2)} \edag )      ( I_{a^+}^{1-\vec \alpha, \vec \sigma, \phi}  F)
(Z,W)       d\rho_{\vartheta  \varphi }(Z) = \\
& \int_{\Lambda} \left\{  \frac{\partial         }{\partial Z_{\vartheta  \varphi } }\left[( e^{   \lambda_1(z_1) } \e +
e^{   \lambda_2(z_2) } \edag )      ( I_{a^+}^{1-\vec \alpha, \vec \sigma, \phi} \right]   
 \right. + \\
&  A_{\vartheta  \varphi }   ( e^{   \lambda_1(z_1) } \e +
e^{   \lambda_2(z_2) } \edag )      ( I_{a^+}^{1-\vec \alpha, \vec \sigma, \phi}  F)
(Z,W)  + B_{\vartheta  \varphi }  {\bf i}   ( e^{   \lambda_1(z_1) } \e + \\
& \left. e^{   \lambda_2(z_2) } \edag )   ( I_{a^+}^{1-\vec \alpha, \vec \sigma, \phi}  F)
(Z,W) \right\} dZ \wedge dZ^{*} .
\end{align*}
This establishes the formula combined with \eqref{equa10}.
\end{proof}
We are thus led to a $(\vartheta,\varphi)-$weighted $(\vec{\alpha},\sigma,\phi)-$Cauchy type formula for bicomplex functions
\begin{corollary}
Let $F\in AC^1 (J_P^Q, \bc) $ be as above. It is required that 
$$\dfrac{\partial ^{\vec \alpha,\sigma,\phi}F(Z,W)}{\partial Z_{\vartheta \varphi, a^+}}= 0$$ 
for all  $Z\in J_P^Q$. Then 
\begin{align*}   
\int_{\gamma} ( e^{\lambda_1 } \e + e^{\lambda_2 } \edag ) ( I_{a^+}^{1-\vec \alpha, \vec \sigma, \phi}  F)(Z,W) d\rho_{\vartheta  \varphi }(Z) 
= \int_{\Lambda} \left[   A_{\vartheta  \varphi }   ( e^{   \lambda_1 } \e + e^{\lambda_2 } \edag ) ( I_{a^+}^{1-\vec \alpha, \vec  \sigma, \phi}  F)(Z,W) \right. \\
+ B_{\vartheta  \varphi }  {\bf i}   ( e^{   \lambda_1 } \e + e^{\lambda_2 } \edag )( I_{a^+}^{1-\vec  \alpha, \vec \sigma, \phi} F)
(Z,W)] dZ \wedge dZ^{*}.
\end{align*}
\end{corollary} 

\begin{theorem}\label{weighted_Borel-Pompeiu_Bicomplex} 
Under the same hypotheses, for $W\in \Lambda$ fixed and $F\in AC^1 (J_P^Q, \bc)$ such that 
$$(e^{\lambda_1 } \e + e^{   \lambda_2 } \edag )( I_{a^+}^{1-\vec \alpha,\vec  \sigma, \phi}  F)(\cdot,W)\in C^1(J_P^Q , \bc) \cap  C(\overline{J_P^Q}, \bc).$$
Then 
\begin{align*} 
  & \left( f_{1}(x_{1}+\li \Im w_{1}) +   f_{1}(\Re w_{1}+\li y_{1})       
 \right)  \e +  \left( f_{2}(x_{2}    +\li \Im w_{2}) +  f_{2}(\Re w_{2}+\li y_{2}) \right) \edag =\\
& \int_{\gamma  }  
 \mathcal  E^{\vec \alpha,\vec  \sigma, \phi}_{(\vartheta, \varphi)} (V , Z )
   ( I_{a^+}^{1-\vec \alpha, \vec \sigma, \phi}  F)
(V,W)      d\rho_{(\vartheta, \varphi)}(V ) \  - \  R^{\vec \alpha, \vec \sigma, \phi}F(Z,W) \ - \\
& D_{a^+}^{1-\vec \alpha,\vec  \sigma, \phi}  \int_{\Lambda} ( e^{   \lambda_1(v_1)-   \lambda_1(z_1) } \e +
e^{   \lambda_2 (v_2) -  \lambda_2(z_2)  } \edag )
E_{(\vartheta , \varphi)}(V ,Z )   (  D \phi (V) )  \\
& \hspace{2cm} \sigma^{-1}   
 \frac{\partial ^{
\vec  \alpha,\sigma,\phi}F(V,W)}{\partial V_{\vartheta \varphi, a^+ } }
dV\wedge dV^{*},
\end{align*}
where 
$$\mathcal  E^{\vec \alpha, \vec \sigma, \phi}_{(\vartheta, \varphi)} (V , Z ):=
D_{a^+}^{1-\vec \alpha,\vec  \sigma, \phi}\left[    ( e^{   \lambda_1 (v_1)  - \lambda_1 (z_1)   } \e +
e^{\lambda_2 (v_2)-\lambda_2 (z_2)} \edag) E_{(\vartheta, \varphi)} (V , Z )\right]$$
and the parameters of operator $D_{a^+}^{1-\vec \alpha,\vec  \sigma, \phi}$ are the real components of $Z$.
\end{theorem}

\begin{proof} Application of Theorem \ref{weighted_Borel-Pompieu_Bicomplex} to the function $(e^{\lambda_1 } \e +
e^{\lambda_2 }\edag) (I_{a^+}^{1-\vec \alpha, \vec \sigma, \phi} F)(\cdot ,W)$ yields
\begin{align*} &  ( e^{   \lambda_1(z_1) } \e +
e^{   \lambda_2(z_2) } \edag )      ( I_{a^+}^{1-\vec \alpha, \vec \sigma, \phi}  F)
(Z,W)   c_{(\vartheta,\varphi)} \\
 = &
\int_{\gamma  }  
    ( e^{   \lambda_1 (v_1)} \e +
e^{   \lambda_2 (v_2)} \edag )      ( I_{a^+}^{1-\vec \alpha,\vec  \sigma, \phi}  F)
(V,W)   E_{(\vartheta, \varphi)} (V , Z ) d\rho_{(\vartheta, \varphi)}(V ) 
 \\
  & - \int_{\Lambda} 
E_{(\vartheta , \varphi)}(V ,Z) \frac{\partial \left[(e^{\lambda_1(v_1) } \e +
e^{   \lambda_2(v_2) } \edag) (I_{a^+}^{1-\vec \alpha,\vec  \sigma, \phi}  F)
(V,W)    \right]}{\partial V_{(\vartheta,\varphi)}} dV\wedge dV^{*},
\end{align*}
where $V=v_1\e + v_2 \edag$.  Then equation \eqref{equa10} allows us to obtain that 
\begin{align*}&   ( e^{   \lambda_1(z_1) } \e +
e^{   \lambda_2(z_2) } \edag )      ( I_{a^+}^{1-\vec \alpha, \vec \sigma, \phi}  F)
(Z,W)   c_{(\vartheta,\varphi)}\\
  = &
\int_{\gamma  }  
    ( e^{   \lambda_1 (v_1)} \e +
e^{   \lambda_2 (v_2)} \edag )      ( I_{a^+}^{1-\vec \alpha, \vec \sigma, \phi}  F)
(V,W)   E_{(\vartheta, \varphi)} (V , Z ) d\rho_{(\vartheta, \varphi)}(V ) 
 \\
  & - \int_{\Lambda} 
E_{(\vartheta , \varphi)}(V ,Z ) ( e^{   \lambda_1(v_1) } \e +
e^{   \lambda_2 (v_2)} \edag )  (  D \phi (V) )  \sigma^{-1}   
 \frac{\partial ^{
\vec  \alpha,\sigma,\phi}F(V,W)}{\partial V_{\vartheta \varphi, a^+ } }
dV\wedge dV^{*}.
\end{align*}
Therefore
\begin{align*}
& ( I_{a^+}^{1-\vec \alpha, \vec \sigma, \phi}  F)(Z,W) c_{(\vartheta,\varphi)}\\
= & \int_{\gamma}(e^{   \lambda_1 (v_1)  -   \lambda_1 (z_1)   } \e +
e^{   \lambda_2 (v_2)-   \lambda_2 (z_2)  } \edag )        E_{(\vartheta, \varphi)} (V , Z )   ( I_{a^+}^{1-\vec \alpha, \vec \sigma, \phi}  F)
(V,W) d\rho_{(\vartheta, \varphi)}(V) \\
  & - \int_{\Lambda} ( e^{   \lambda_1(v_1)-  \lambda_1(z_1) } \e +
e^{   \lambda_2 (v_2) -  \lambda_2(z_2)  } \edag )
E_{(\vartheta , \varphi)}(V ,Z )   (  D \phi (V) )  \sigma^{-1}   
 \frac{\partial ^{
\vec  \alpha,\sigma,\phi}F(V,W)}{\partial V_{\vartheta \varphi, a^+ } }
dV\wedge dV^{*}.
\end{align*}
Acting operator $D_{a^+}^{1-\vec \alpha,\vec  \sigma, \phi}$ on both sides,  using \eqref{equa9} and Leibniz rule we obtain that   
\begin{align*} 
  & \left( f_{1}(x_{1}+\li \Im w_{1}) +   f_{1}(\Re w_{1}+\li y_{1})       
 \right)  \e +  \left( f_{2}(x_{2}    +\li \Im w_{2}) +  f_{2}(\Re w_{2}+\li y_{2}) \right) \edag\\
  = &
\int_{\gamma  }  
D_{a^+}^{1-\vec \alpha, \vec \sigma, \phi}\left[    ( e^{   \lambda_1 (v_1)  -   \lambda_1 (z_1)   } \e +
e^{   \lambda_2 (v_2)-   \lambda_2 (z_2)  } \edag )        E_{(\vartheta, \varphi)} (V , Z )\right] 
   ( I_{a^+}^{1-\vec \alpha, \vec \sigma, \phi}  F)
(V,W)      d\rho_{(\vartheta, \varphi)}(V ) 
 \\
  & -  D_{a^+}^{1-\vec \alpha,\vec  \sigma, \phi}  \int_{\Lambda} ( e^{   \lambda_1(v_1)-  \lambda_1(z_1) } \e +
e^{   \lambda_2 (v_2) -  \lambda_2(z_2)  } \edag )
E_{(\vartheta , \varphi)}(V ,Z )   (  D \phi (V) )  \sigma^{-1}   
 \frac{\partial ^{
\vec  \alpha,\sigma,\phi}F(V,W)}{\partial V_{\vartheta \varphi, a^+ } }
dV\wedge dV^{*}\\
& - R^{\vec \alpha,\vec  \sigma, \phi}  F(Z,W) .
\end{align*}
\end{proof}
 
\begin{corollary} Under the hypothesis and notation of the previous theorem.  
 If $$  \frac{\partial ^{
 \vec \alpha,\sigma,\phi}f(V,W)}{\partial V_{\vartheta \varphi, a^+ } } =0 ,\quad  V\in J_P^Q.$$ Then 
\begin{align*} 
  & \left( f_{1}(x_{1}+\li \Im w_{1}) +   f_{1}(\Re w_{1}+\li y_{1})       
 \right)  \e +  \left( f_{2}(x_{2}    +\li \Im w_{2}) +  f_{2}(\Re w_{2}+\li y_{2}) \right) \edag\\
  = & \int_{\gamma} \mathcal  E^{\vec \alpha, \vec \sigma, \phi}_{(\vartheta, \varphi)} (V , Z )( I_{a^+}^{1-\vec \alpha, \vec \sigma, \phi}  F)
(V,W)\ d\rho_{(\vartheta, \varphi)}(V ) - R^{\vec \alpha,\vec  \sigma, \phi} F(Z,W).
\end{align*} 
\end{corollary}
\section*{Concluding remarks}
In conclusion, the following facts hold:
\begin{enumerate}
\item For $\sigma =1$ and  $\phi(Z)= \phi(z_1\e + z_2\edag)=x_1+ y_1+x_2+y_2  $ on $J_P^Q$ we see that $\displaystyle \frac{\partial ^{
\vec  \alpha,\sigma,\phi}}{\partial Z_{\vartheta \varphi, a^+ } }$ coincides with 
 the bicomplex weighted  fractional Cauchy-Riemann type operator $\displaystyle\frac{\partial ^{
\vec  \alpha}}{\partial Z_{\vartheta \varphi, a^+ } }$ studied in \cite{BG} and Theorems \ref{weighted_Gauss_Bicomplex2}  and  \ref{weighted_Borel-Pompeiu_Bicomplex} extend, preserving the  structure, Theorems  4.1. and  4.3. respectively of  \cite{BG}.

\item Given $0<\delta_0,\delta_1,\delta_2,\delta_3<1$. If   $ \alpha_0= \alpha_1=\alpha_2=\alpha_3=1$, $\sigma =1$ and  $\phi(Z)= \phi(z_1\e + z_2\edag)=x_1^{\delta_0}+ y_1^{\delta_1} +x_2^{\delta_2}+y_2^{\delta_3}  $ on $J_P^Q$ then $\displaystyle\frac{\partial ^{
\vec  \alpha,\sigma,\phi}}{\partial Z_{\vartheta \varphi, a^+ } }$ is a  
 the bicomplex weighted  fractal  Cauchy-Riemann type operator $\frac{\partial ^{
 \alpha}}{\partial Z_{\vartheta \varphi, a^+ } }$ extending the function theory commented in before. So for 
the bicomplex weighted  fractal function theory  we  have,  as   particular consequences,   Stokes and Borel-Pompeiu type formulas and then the Cauchy Theorem and Formula. 
\item Given $0<\delta_0,\delta_1,\delta_2,\delta_3<1$. If   $ \alpha_0=
 \alpha_1=\alpha_2=\alpha_3=1$  and  $\phi(Z)= \phi(z_1\e + z_2\edag)=x_1^{\delta_0}+ y_1^{\delta_1} +x_2^{\delta_2}+y_2^{\delta_3}  $ on $J_P^Q$  then     $ \displaystyle\frac{\partial ^{
\vec  \alpha,\sigma,\phi}}{\partial Z_{\vartheta \varphi, a^+ } }$ is a  
 the bicomplex weighted proportional and  fractal  Cauchy-Riemann type operator which induces a natural extension of the function theory commented in the previous case. 
\item Given $0<\delta_0,\delta_1,\delta_2,\delta_3<1$  and  $\phi(Z)=  x_1^{\delta_0}+ y_1^{\delta_1} +x_2^{\delta_2}+y_2^{\delta_3}  $ on $J_P^Q$  then     $ \displaystyle\frac{\partial ^{
\vec  \alpha,\sigma,\phi}}{\partial Z_{\vartheta \varphi, a^+ } }$ is a the bicomplex weighted fractional and  fractal  Cauchy-Riemann type operator.
\end{enumerate}

\section*{Statements and Declarations}
\subsection*{Funding} This work is supported in part by Instituto Polit\'ecnico Nacional (grant numbers SIP20232103, SIP20230312) and CONACYT.
\subsection*{Competing Interests} The authors have no known competing financial interests or personal relationships that could have appeared to influence the work reported in this paper.
\subsection*{Author contributions} All authors accept responsibility for the entire content of this manuscript and approved its final version. 
\subsection*{Availability of data and material} Not applicable
\subsection*{Code availability} Not applicable
\subsection*{ORCID}
\noindent
Jos\'e Oscar Gonz\'alez-Cervantes: https://orcid.org/0000-0003-4835-5436\\
Juan Bory-Reyes: https://orcid.org/0000-0002-7004-1794

\end{document}